\documentclass[12pt]{amsart}
\usepackage{amsmath,latexsym,amsfonts,amssymb,amsthm}
\usepackage{geometry}
\usepackage{hyperref}
\usepackage{mathrsfs}
\usepackage{graphicx,color}
\usepackage{tikz-cd}
\usepackage{float}

\newtheorem{prop}{Proposition}
\newtheorem{thm}[prop]{Theorem}

\newtheorem{conj}[prop]{Conjecture}
\newtheorem{lem}[prop]{Lemma}

\theoremstyle{definition}

\newtheorem{defn}[prop]{Definition}
\newtheorem{expl}[prop]{Example}
\newtheorem{rem}[prop]{\it Remark}
\newtheorem{emp}[prop]{}

\numberwithin{equation}{section}

\newcommand{\bP}{\mathbb{P}}
\newcommand{\bC}{\mathbb{C}}
\newcommand{\bR}{\mathbb{R}}
\newcommand{\bA}{\mathbb{A}}
\newcommand{\bZ}{\mathbb{Z}}

\newcommand{\oY}{\overline{Y}}

\newcommand{\cX}{\mathcal{X}}
\newcommand{\cY}{\mathcal{Y}}

\newcommand{\cO}{\mathcal{O}}
\newcommand{\cL}{\mathcal{L}}
\newcommand{\cI}{\mathcal{I}}
\newcommand{\cM}{\mathcal{M}}

\newcommand{\tX}{\widetilde{X}}
\newcommand{\tY}{\widetilde{Y}}

\newcommand{\tphi}{\tilde{\phi}}

\newcommand{\tsigma}{\tilde{\sigma}}
\newcommand{\tmu}{\tilde{\mu}}
\newcommand{\tpi}{\tilde{\pi}}
\newcommand{\hrho}{\hat{\rho}}

\newcommand{\tcL}{\widetilde{\mathcal{L}}}
\newcommand{\tcX}{\widetilde{\mathcal{X}}}
\newcommand{\tcY}{\widetilde{\mathcal{Y}}}
\newcommand{\hX}{\hat{X}}

\newcommand{\Proj}{\mathbf{Proj}}
\newcommand{\Supp}{\mathrm{Supp}}

\newcommand{\chitop}{\chi_{\mathrm{top}}}

\makeatletter
\def\namedlabel#1#2{\begingroup
   \def\@currentlabel{#2}%
   \label{#1}\endgroup
}
\makeatother

\begin{document}

\title{Hyperbolicity of Cyclic Covers and Complements}
\author{Yuchen Liu}
\address{Department of Mathematics, Princeton University, Princeton, NJ, 08544-1000, USA.}
\email{yuchenl@math.princeton.edu}
\thanks{The author is partially supported by NSF grant DMS-0968337.}

\begin{abstract}
We prove that a cyclic cover of a smooth
complex projective variety is Brody hyperbolic if its branch 
divisor is a generic small deformation of a large enough multiple of a 
Brody hyperbolic base-point-free ample divisor.
We also show the hyperbolicity of complements of those branch divisors.
As an application, we find new examples of Brody hyperbolic hypersurfaces in $\bP^{n+1}$ 
that are cyclic covers of $\bP^n$.
\end{abstract}

\maketitle

\section{Introduction}

A complex analytic space $X$ is called \textit{Brody hyperbolic} if there 
are no non-constant holomorphic maps from $\bC$ to $X$. 
Lang's conjecture \cite{lang} predicts that a projective variety $X$ is Brody 
hyperbolic if every subvariety
 of $X$ is of general type. 
More generally, the Green-Griffiths-Lang conjecture \cite{GG, lang}
predicts that if a projective variety $X$ is of general type, then
there exists a proper Zariski closed subset $Z\subsetneq X$ such
that any non-constant holomorphic map $f:\bC\to X$ satisfies 
$f(\bC)\subset Z$. For hyperbolicity of hypersurfaces, Kobayashi
\cite{kob70, kob98} proposed the following conjecture:

\begin{conj}[Kobayashi]\label{kobconj}
For $n\geq 3$, a general hypersurface 
$X\subset\bP^n$ of degree $\geq (2n-1)$ is Brody hyperbolic.
\end{conj}

It is easy to see that Lang's 
conjecture follows from the Green-Griffiths-Lang conjecture by
a Noetherian induction argument.
Based on results by Clemens \cite{cle86}, Ein \cite{ein88, ein91}
and Xu \cite{xu94}, Voisin \cite{voi96} showed that a general
hypersurface $X$ of degree $\geq (2n-1)$ in $\bP^n$ with $n\geq 3$ satisfies that
every subvariety of $X$ is of general type. Therefore,
Lang's conjecture implies Conjecture \ref{kobconj} by Voisin's result. 

A lot of work has been done toward Conjecture \ref{kobconj}, see
\cite{mcq99, deg00, rou07, pau08, dmr10, siu15, dem15, bro16, den16}. 
Examples of hyperbolic hypersurfaces are constructed in 
\cite{mn96, sy97, fuj01, sz02, cz03, duv04, sz05, cz13, huy15, huy16}.


In this paper, we first study the hyperbolicity of cyclic covers. For curves we know that if the branch divisor has large degree, then the cyclic cover will be Brody hyperbolic for generic choice of the branch divisor. (See Section \ref{p1} for hyperbolicity of cyclic covers of $\bP^1$.) Our first main result is a higher dimensional generalization:

\begin{thm}\label{mainthm}
 Let $X$ be a smooth projective variety with $\dim X=n$. Let $L$ be a globally
 generated ample line bundle on $X$.
 Suppose that there exists a smooth hypersurface $H\in |L|$ that is Brody hyperbolic.
 Let $m,d\geq 2$ be positive integers such that $m$ is a multiple of $d$.
 For a generic small deformation $S$ of
 $mH\in |L^{\otimes m}|$, let $Y$ be the degree $d$ cyclic cover of $X$ branched along $S$. 
 Then $Y$ is Brody hyperbolic if $m\geq d\lceil \frac{n+2}{d-1} \rceil$.
\end{thm}

Here the assumption that $H$ being Brody hyperbolic is crucial for our discussion.
If $X\setminus H$ is also Brody hyperbolic, we have slightly better lower bounds on $m$ to settle the hyperbolicity of $Y$:

\begin{thm}\label{mainthm2}
 With the notation of Theorem \ref{mainthm}, assume in addition that
 $X\setminus H$ is Brody hyperbolic. Then $Y$ is Brody
hyperbolic if $m\geq d\lceil\frac{n+1}{d-1}\rceil$.
\end{thm}

In fact, we prove stronger results from which Theorem \ref{mainthm} and \ref{mainthm2} follow (see Theorem
\ref{refined1} and \ref{refined2}).  
\medskip

The following theorem is an application of Theorem \ref{mainthm} which gives new examples of Brody hyperbolic hypersurfaces in $\bP^{n+1}$. These hypersurfaces are cyclic covers of $\bP^n$ via linear projections.

\begin{thm}\label{cyclichyp}
Suppose $D$ is a smooth hypersurface of degree $k$ in $\bP^n$ that is Brody hyperbolic. For $d\geq n+3$ and a generic small deformation $S$ of $d D$ in $|\cO_{\bP^n}(dk)|$, let $W$ be the degree $dk$ cyclic cover of $\bP^n$ branched along $S$. Then $W$ is a Brody hyperbolic hypersurface in $\bP^{n+1}$ of degree $dk$.
\end{thm}

In \cite{rr13}, Roulleau-Rousseau showed that a double cover of
$\bP^2$ branched along a very general curve of degree at least
$10$ is algebraically hyperbolic. Thus Green-Griffiths-Lang conjecture predicts that these surfaces are also Brody hyperbolic. As an application of Theorem \ref{mainthm2}, we give some evidence supporting this prediction:

\begin{thm}\label{ccp2}
Let $l\geq 3$, $k\geq 5$ be two positive integers. Let $D$ be a smooth plane curve of degree $k$ such that $\bP^2\setminus D$ is Brody hyperbolic. (The existence of such $D$ was shown by Zaidenberg in \cite{zai89}.) Let $S$ be a generic small deformation of $2l D$. Then the double cover of $\bP^2$ branched along $S$ is Brody hyperbolic.
\end{thm}

Note that the minimal degree of $S$ is $30$.

\medskip

For hyperbolicity of complements, the logarithmic Kobayashi 
conjecture and related problems have been studied in
\cite{gre77, zai87, zai89, sy96, sz00, eg03, rou07b, rou09, it15}.

The cyclic cover being Brody hyperbolic clearly implies that the complement of the branch locus is also Brody hyperbolic. More precisely, with the notation of Theorem \ref{mainthm} we have that $X\setminus S$ is Brody hyperbolic if $m\geq d\lceil \frac{n+2}{d-1} \rceil$. In fact, we can still reach the same conclusion with the slightly weaker condition $m\geq n+2$, as the following theorem states.

\begin{thm}\label{complement}
Let $X$ be a smooth projective variety with $\dim X=n$. Let $L$ be a globally
 generated ample line bundle on $X$.
 Suppose that there exists a smooth hypersurface $H\in |L|$ that is Brody hyperbolic.
 Let $m\geq n+2$ be a positive integer. Then
 for a generic small deformation $S$ of
 $mH\in |L^{\otimes m}|$, both $S$ and $X\setminus S$ are Brody hyperbolic.
 Moreover, $X\setminus S$ is complete hyperbolic and hyperbolically
 embedded in $X$.
\end{thm}

Recall that \textit{complete hyperbolicity} is defined in \cite[p. 60]{kob98} and \textit{hyperbolically embeddedness} is defined in \cite[p. 70]{kob98}.

\subsection*{Structure of the paper}
The proofs of the theorems are mostly based on the degeneration to the normal cone (Section \ref{normalcone}) and deformation type theorems of hyperbolicity (Theorem \ref{deformation}, \ref{sz} and \ref{def-comp}). 

In Section \ref{family}, we construct a family $\cX\to\bA^1$ with the general fiber $X_t$ isomorphic to $X$ and the special fiber $X_0$ being a projective cone over $H$. For technical reasons, we first 
introduce a smooth model $\tcX$ of $\cX$ (Proposition \ref{tcX}). Then by taking a cyclic cover of the total space $\cX$, we get a family of cyclic covers $Y_t$ of $X$ which degenerates to a cyclic cover $Y_0$ of $X_0$ (Proposition \ref{tcY}). 

Most of the theorems are proved in Section \ref{bigpf}. As a preparation, we study the hyperbolicity of cyclic covers of $\bP^1$ and $\bA^1$ in Section \ref{p1}. We give lower bounds on the size of the reduced branch loci to get hyperbolicity of those cyclic covers (Lemma \ref{fiberhyp1} and \ref{fiberhyp2}).
Section \ref{pf} is devoted to proving Theorem \ref{mainthm}, \ref{mainthm2} and \ref{complement}. By dimension counting, we give a lower bound for the size of the reduced branch loci among all generators of $X_0$ (Lemma \ref{tangentdim}). When the branch locus has large degree, we prove the hyperbolicity of the cyclic cover of each generator of $X_0$,
which gives the hyperbolicity of $Y_0$ since the base $H$ is also Brody hyperbolic. Then Theorem \ref{mainthm} follows by applying a deformation type theorem (see Theorem \ref{deformation}). Theorem \ref{mainthm2} and \ref{complement} are proved in similar ways with minor changes.

We first apply our methods to hypersurfaces in $\bP^n$ in Section \ref{ex-hyp}.
We prove a stronger result that essentially implies Theorem 
\ref{cyclichyp} (see Theorem \ref{hcc}). We give a new proof
to the main result in \cite{zai09} by applying Mori's degeneration
method (see Theorem \ref{zaicone}). We also improve
\cite[p. 147, Corollary of Theorem II.2]{zai92} (see Theorem 
\ref{hypersurface}). In Section \ref{ex-surface}, we apply our methods to surfaces. We prove Theorem \ref{ccp2}. We also obtain hyperbolicity of the complement of a smooth curve in some polarized K3 surface of large degree (Example \ref{k3}).

\subsection*{Notation}
Throughout this paper, we work over the complex numbers $\bC$. We will follow the 
terminology of \cite{kob98} for various notions of hyperbolicity.

\subsection*{Acknowledgement}
I would like to thank my advisor J\'anos Koll\'ar for his constant support, encouragement and many inspiring conversations.
I wish to thank Chi Li, Charles Stibitz, Amos Turchet and
Anibal Velozo for many helpful discussions, and Gang Tian, 
Xiaowei Wang and Chenyang Xu for their interest and encouragement.
I also wish to thank Mikhail Zaidenberg for his helpful comments 
through e-mails.

\section{Construction of Families}\label{family}

\subsection{Degeneration to the normal cone}\label{normalcone}

From now on $X$ will be a 
smooth projective variety of dimension $n$. 
Let $L$ be a globally generated ample line
bundle on $X$. Let $H$ be a smooth hypersurface in $|L|$.

Let $\rho:\tcX\to X\times\bA^1$ be the blow up of $X\times\bA^1$ along $H\times\{0\}$, with exceptional divisor $E$.
Denote the two projections from $X\times\bA^1$ by $p_1$ and $p_2$. The line bundle $\tcL$ on $\tcX$ is defined by
\[
\tcL:= (p_1\circ\rho)^* L\otimes \cO_{\tcX}(-E).
\]
Let $\tpi:=p_2\circ\rho$ be the composite of the projections $\tcX\to X\times\bA^1\to\bA^1$.

\begin{prop}
With the above notation, the line bundle $\tcL$ is globally generated.
\end{prop}

\begin{proof}
Let $\cI$ be the ideal sheaf of $H\times\{0\}$ in $X\times\bA^1$. Then  $\cO_{\tcX}(-E)=\rho^{-1}\cI\cdot \cO_{\tcX}$. Therefore, it suffices to show that $p_1^*L\cdot\cI$ is globally generated.

Let us choose a basis $s_1,\cdots, s_N$ of the vector space $H^0(X,L)$ with $N:=\dim H^0(X,L)$. Let $s_H\in H^0(X,L)$ be a defining section of $H$, i.e. $H=(s_H=0)$. We may define sections $\sigma_0,\sigma_1,\cdots,\sigma_N\in H^0(X\times\bA^1, p_1^*L)$ as follows:
\begin{align*}
\sigma_0(x,t)& =s_H(x),\\
\sigma_i(x,t)& =t s_i(x)\quad\textrm{for any }1\leq i\leq N.
\end{align*}

Since $s_1,\cdots,s_N$ generate $L$, the sections $\sigma_0,\cdots,\sigma_N$ generate the subsheaf $p_1^*L\cdot\cI$ of $p_1^*L$. Hence we prove the proposition.
\end{proof}

Denote the lifting of $\sigma_i$ to $\tcX$ by $\tsigma_i\in H^0(\tcX,\tcL)$. The proof above implies that $\tphi:=[\tsigma_0,\cdots,\tsigma_N]$ defines a morphism $\tphi:\tcX\to\bP^N$, such that $\tcL\cong \tphi^*\cO(1)$.

Since $\tcL$ is globally generated, it is also $\tpi$-globally generated. By \cite[2.1.27]{laz04}, we may define the ample model $\cX$ of $(\tcX,\tcL)$ over $\bA^1$ by
\[
\cX:=\Proj_{\bA^1}\bigoplus_{i\geq 0}\tpi_*(\tcL^{\otimes i}),
\]
where $\psi:\tcX\to\cX$ is an algebraic fibre space with $\cX$ normal. Then $\tcL$ descends to a globally generated ample line bundle $\cL$ on $\cX$, i.e. $\tcL=\psi^*\cL$. Since $\tphi$ is induced by a base point free sub linear system of $|\tcL|$, $\tphi$ descends to a morphism $\phi:\cX\to\bP^N$, i.e. $\tphi=\phi\circ\psi$.

Denote the fibers of $\tpi:\tcX\to\bA^1$ and $\pi:\cX\to\bA^1$ by
$\tX_t$ and $X_t$, respectively. Then as a Cartier
divisor, $\tX_0$ can be written as
\[\tX_0= \hX_0+E,\] where $\hX_0$ is the birational 
transform of $X\times\{0\}$ under $\rho$.

\begin{prop}\label{tcX} With the above notation, we have the following properties.
\begin{enumerate}
\item The Stein factorization of $(\tphi,\tpi):\tcX\to\bP^N\times\bA^1$ is given by the following commutative diagram
 \[
 \begin{tikzcd}
  \tcX\arrow{r}{\psi}\arrow{rd}[swap]{(\tphi,\tpi)} & 
  \cX\arrow{d}
  {(\phi,\pi)}\\
  & \bP^N\times\bA^1
 \end{tikzcd}
 \]
where $\psi_*\cO_{\tcX}=\cO_{\cX}$, and $(\phi,\pi)$ is finite. 
\item The morphism $\psi$ is birational. More precisely, $\psi$ is an isomorphism away from
$\hX_0$, and it contracts $\hX_0$ to a point $v_0$ in $\cX$.
\item\label{tcX3} Let $\rho_E:E\to H$ be the $\bP^1$-bundle structure on $E$.
Then $\tphi$ sends each fiber of $\rho_E$ isomorphically onto a line in $\bP^N$.
\end{enumerate}
\end{prop}

\begin{proof}
\hfill\par
(1) It follows from the relative version of \cite[2.1.28]{laz04}. 
\medskip

(2) Since $\rho:\tcX\setminus\tX_0\to X\times(\bA^1\setminus\{0\})$ is an isomorphism, the restriction $\tcL|_{\tcX\setminus\tX_0}$ is $\tpi$-ample over $\bA^1\setminus\{0\}$. Hence $\psi$ is an isomorphism away from $\tX_0$, which implies that $\psi$ is birational.

Let $\hrho_0:=(p_1\circ\rho)|_{\hX_0}$
be the isomorphism from $\hX_0$ to $X$. Recall that $\tcL= (p_1\circ\rho)^* L\otimes \cO_{\tcX}(-E)$. Then
\begin{align*}
\tcL|_{\hX_0}& \cong \hrho_0^* L\otimes \cO_{\hX_0}(-E|_{\hX_0})\\
& \cong \hrho_0^* L\otimes\hrho_0^* \cO_X(-H)\\
& \cong \hrho_0^* (L\otimes\cO_X(-H))\\
& \cong \cO_{\hX_0}
\end{align*}
Hence $\tcL|_{\hX_0}$ is trivial, which implies that $\psi$ contracts $\hX_0$ to a point $v_0$ in $\cX$.

Since $N_{H/X}\cong\cO_X(H)|_H\cong L|_H$, we have $E\cong \bP_H(L|_H^\vee\oplus\cO_H)$.
It is clear that $H_0:=\hX_0|_E$ is the section of $\rho_E$
corresponding to the first projection $L|_H\oplus\cO_H\to L|_H$.
Denote by $H_1$ the other section of $\rho_E$ corresponding to the second
projection $L|_H\oplus\cO_H\to\cO_H$. Then
\begin{align*}
 \tcL|_E&\cong(p_1\circ\rho)^*(L)|_E\otimes\cO_{\tcX}(-E)|_E\\
 &\cong \rho_E^*(L|_H)\otimes\cO_{\tcX}(\hX_0)|_E\otimes
 \cO_{\tcX}(\tX_0)|_E\\
 &\cong \rho_E^*(L|_H)\otimes\cO_E(H_0)\\
 &\cong \cO_E(H_1)
\end{align*}
Since $L|_H$ is ample, for sufficiently large $k$ the linear system 
$|\tcL|_E^{\otimes k}|$ gives a birational morphism $E\to C_p(H,L|_H)$, 
where $C_p(H,L|_H)$ is the projective cone in the sense of
\cite[Section 3.1]{kol13}. In particular, any curve contracted by
$\psi|_E$ is contained in $H_0$. Thus $\psi|_E$ is an isomorphism away
from $H_0$, we prove (2).
\medskip

(3) As we have seen in the proof of (2), $\tcL|_E\cong\cO_E(H_1)$. Hence 
$(\tphi^*\cO(1)\cdot\rho_E^{-1}(x))=(\tcL\cdot\rho_E^{-1}(x))=1$ for any $x\in H$, we prove (3).
\end{proof}

\begin{rem}
 By the proof of Proposition \ref{tcX}, the Stein factorization
 of $\psi|_E:E\to X_0$ is given by $E\to C_p(H,L|_H)\to X_0$.
 Thus $C_p(H,L|_H)$ is 
 isomorphic to the normalization $X_0^\nu$ of $X_0$. In general, $X_0$
 is not necessarily normal. According to \cite[3.10]{kol13}, $X_0$ is 
 normal if and only if $H^1(X,L^{\otimes k})=0$ for any $k\geq 0$.
\end{rem}

\subsection{Constructing families of cyclic covers}

Let $m,d\geq 2$ be positive integers such that $m$ is a multiple
of $d$. It is clear that the linear system 
$\phi^*(|\mathcal{O}_{\bP^N}(m)|)$ is base point free. 
Hence by Bertini's theorem, the following
property holds for a 
general hypersurface $T\in |\mathcal{O}_{\bP^N}(m)|$:

\begin{enumerate}
 \item[($*$)\namedlabel{star}{($*$)}] \qquad$\phi^*(T)$ is smooth, does not contain $v_0$
 and intersects $X_0$ transversally.
\end{enumerate}

Fix a general hypersurface $T\in |\mathcal{O}_{\bP^N}(m)|$ 
satisfying property
\ref{star}. Let $\mu:\cY\to\cX$ and $\tmu:\tcY\to\tcX$ be the 
degree $d$ cyclic covers of $\cX$ and $\tcX$ branched along
$\phi^*(T)$ and $\tphi^*(T)$, respectively. 
Let $\pi_{\tcY}:=\tpi\circ\tmu$ and $\pi_{\cY}:=\pi\circ\mu$
be the composition maps.
Let $\psi_1:\tcY\to\cY$ be the lifting of $\psi:\tcX\to\cX$.
Then $\tcY$ and $\cY$ are proper flat families over $\bA^1$.
Denote the fibers of $\pi_{\tcY}$ and $\pi_{\cY}$ by $\widetilde{Y}_t$ and $Y_t$, respectively. 

For $t\neq 0$ we notice that $(p_1\circ\rho)|_{\tX_t}$ maps $\tX_t$ isomorphically onto $X$. Hence we may define a family of maps $f_t:X\to\bP^N$ by $f_t:=\tphi\circ(p_1\circ\rho)|_{\tX_t}^{-1}$ for $t\neq 0$. In projective coordinates, we have
\[
f_t(x)=[s_H(x), t s_1(x),\cdots,t s_N(x)].
\]

Denote by $S_t:=f_t^*(T)$ the pull back of $T$ under $f_t$. Then $S_t\in|L^{\otimes m}|$ for all but finitely many $t$. 
In projective coordinates, let $F=F(z_0,\cdots,z_N)$ be a degree $m$ homogeneous polynomial such that $T=(F=0)$. Expand $F$ as a polynomial of the single variable $z_0$ yields 
\[
F(z_0,z_1,\cdots,z_N)=F_0 z_0^m+F_1 z_0^{m-1}+\cdots+F_{m-1} z_0 +F_m,
\]
where $F_i$ is a homogeneous polynomial in $z_1,\cdots,z_N$ of degree $i$ for $0\leq i\leq m$.
Then $S_t$ is the zero locus of the following section in $H^0(X,L^{\otimes m})$:
\[
F(s_H, ts_1,\cdots,ts_N)=F_0 s_H^m+ \sum_{i=1}^m F_i(s_1,\cdots,s_N) s_H^{m-i} t^i.
\]

Since $T$ does not contain $\phi(v_0)=[1,0,\cdots,0]$, we have that $F_0\neq 0$. 
For simplicity we may assume that $F$ is a monic polynomial in $z_0$, i.e. $F_0=1$.
Thus for any $t\neq 0$ we have
\[
S_t=\left(s_H^m+ \sum_{i=1}^m F_i(s_1,\cdots,s_N) s_H^{m-i} t^i=0\right).
\]

\begin{defn}\label{generic}
 With the above notation, we say that $S\in |L^{\otimes m}|$ is a 
 \textit{generic small deformation} of $mH$, if $S$ is the zero locus of the section
 \[
 s_H^m + \sum_{i=1}^m F_i(s_1,\cdots,s_N) s_H^{m-i} t^i
 \]
 for generic choices of degree $i$ polynomials $F_i$ and for some $t\in\bA^1\setminus\{0\}$ with $|t|\leq \epsilon$, where 
 $\epsilon=\epsilon(\{F_i\})\in\bR_{>0}$ depends on the choice of $\{F_i\}$.
\end{defn}

Notice that Definition \ref{generic} does not depend on the choice
of the basis $s_1,\cdots,s_N$.

From our constructions we see that $S_t$ is automatically a generic small deformation of $mH$ for $|t|$ sufficiently small.

\begin{prop}\label{tcY}
 With the above notation, we have the following properties.
 \begin{enumerate}
  \item The variety $\tcY$ is smooth. 
  
  \item 
  The birational morphism $\psi_1:\tcY\to\cY$ is an isomorphism away from $\tmu^{-1}(\hX_0)$, where $\tmu^{-1}(\hX_0)$ is a disjoint union of $d$ isomorphic copies $\hX_{0,1},\cdots,\hX_{0,d}$ of $\hX_0$. Besides, $\psi_1$ contracts $\hX_{0,i}$ to a point $q_i$ in $\cY$, with $\mu^{-1}(v_0)=\{q_1,\cdots,q_d\}$.
  
  \item Let $\oY_0:=\tmu^{-1}(E)$. Then $\psi_1|_{\oY_0}:\oY_0\to Y_0$ is birational. Besides, the irreducible components of $\tY_0$ are $\oY_0, \hX_{0,1},\cdots,\hX_{0,d}$.

  \item For any $t\neq0$, the fibers $\tY_t$ and $Y_t$ are both isomorphic to the degree $d$
  cyclic cover of $X$ branched along $S_t$, where $S_t$ is a generic small deformation of $mH$ for $|t|$ sufficiently small. Besides, $S_t$ is smooth for $t$ sufficiently small.
 \end{enumerate}
\end{prop}

\begin{proof}
\hfill
\par (1) Since $T$ satisfies \ref{star}, $\tphi^*(T)$ is smooth. Hence $\tcY$ is smooth.

(2) It follows from Proposition \ref{tcX}, $v_0\not\in\phi^*(T)$ and its equivalent form $\tphi^*(T)\cap \hX_0=\emptyset$.

(3) It follows from (2).

(4) From our constructions we see that $(p_1\circ\rho)|_{\tX_t}$ maps $(\tX_t, \tphi^*(T)|_{\tX_t})$ isomorphically onto
$(X, S_t)$, hence the first statement follows. Since $T$ satisfies \ref{star}, $\phi^*(T)$ intersects $X_0$ transversally.
In particular, $\phi^*(T)|_{X_0}$ is smooth. Hence $S_t\cong
\tphi^*(T)|_{\tX_t}\cong \phi^*(T)|_{X_t}$ is smooth for all but finitely many $t$, in particular for $|t|$ sufficiently small.
\end{proof}

\section{Proofs of the Theorems}\label{bigpf}

\subsection{Hyperbolicity of cyclic covers of \texorpdfstring{$\bP^1$}{P1} and \texorpdfstring{$\bA^1$}{A1}}\label{p1}\hfill

Firstly, let us look at the hyperbolicity of cyclic covers of $\bP^1$.

Let $m,d\geq 2$ be positive integers such that $m$ is a multiple of $d$.
Denote by $f:C\to\bP^1$ the degree $d$ cyclic cover of $\bP^1$ branched along an effective divisor $D$ of degree $m$.
We may write
\[
D=a_1 p_1+\cdots + a_l p_l,
\]
where $p_1,\cdots,p_l\in \bP^1$ are distinct closed points, and $\sum_i a_i=m$. 

Let $f^\nu: C^\nu\rightarrow \bP^1$ be the normalization of $f$.
By the Riemann-Hurwitz formula, we have
\[
\chitop(C^\nu)=d\cdot \chitop(\bP^1)-\sum_{i=1}^l (d-\# (f^\nu)^{-1}(p_i)),
\]
where $\chitop(\cdot)$ is the topological Euler number. 
It is clear that $C$ is locally defined by the equation $y^d=x^{a_i}$ near
$p_i$, so $\# (f^\nu)^{-1}(p_i)=\gcd(a_i,d)$. Hence
 \[
  \chitop(C^\nu)=2d-\sum_{i=1}^l (d-\gcd(a_i,d)).
 \]
It is easy to see that $C^\nu$ is a 
 disjoint union of $\gcd(d, a_1,\cdots,a_l)$ isomorphic copies of an irreducible 
 smooth projective curve. Therefore, $C$ is Brody hyperbolic if and only if $\chitop(C^\nu)<0$.

\begin{lem}\label{fiberhyp1}
With the above notation, 
assume in addition that 
one of the following holds:
\begin{enumerate}
 \item[-] $d$ is divisible by $2$ or $3$, and $l\geq \frac{m}{d}+3$;
 \item[-] $d$ is relatively prime to $6$, and $l\geq\frac{m}{d}+2$.
\end{enumerate}
Then $\chitop(C^\nu)<0$, i.e. 
$C$ is Brody hyperbolic.
\end{lem}

\begin{proof}
Assume to the contrary that $\chitop(C^\nu)\geq 0$.
Define an index set $J:=\{j\mid 1\leq j\leq l,\, a_j \textrm{ is not a multiple of }d\}$.
Notice that $\gcd(a_i,d)\leq d/2$ if $a_i$ is not a multiple of $d$.
So we have $\chitop(C^\nu)<0$ as soon as $\#J\geq 5$. 
Hence we only need to consider cases when $\#J\leq 4$. 
For simplicity, we may assume that $J=\{1,\cdots,\#J\}$.

After a careful study we get the following table, which illustrates all cases when $\chitop(C^\nu)\geq 0$, i.e. when
$C$ is not Brody hyperbolic, up to permutations of $a_1,\cdots, a_l$. 

\begin{table}[H]
\[\begin{array}{|c|c|c|c|c|c|c|c|}
\hline \#J & \{a_1/d\} & \{a_2/d\} & \{a_3/d\} & \{a_4/d\} & C_1^\nu & 
\gamma & l\leq\\ 
\hline 0 & 0 & 0 & 0 & 0 & \bP^1 & d & m/d\\
\hline 2 & p/q & (q-p)/q & 0 & 0 & \bP^1 & d/q & m/d+1\\
\hline 3 & 1/2 & 1/3 & 1/6 & 0 & \rm{elliptic} & d/6 & m/d+2\\
 3 & 1/2 & 2/3 & 5/6 & 0 & \rm{elliptic} & d/6 & m/d+1\\
 3 & 1/2 & 1/4 & 1/4 & 0 & \rm{elliptic} & d/4 & m/d+2\\
 3 & 1/2 & 3/4 & 3/4 & 0 & \rm{elliptic} & d/4 & m/d+1\\ 
 3 & 1/3 & 1/3 & 1/3 & 0 & \rm{elliptic} & d/3 & m/d+2\\
 3 & 2/3 & 2/3 & 2/3 & 0 & \rm{elliptic} & d/3 & m/d+1\\
\hline 4 & 1/2 & 1/2 & 1/2 & 1/2 & \rm{elliptic} & d/2 & m/d+2\\
\hline
\end{array}
\]
\caption{Cyclic covers of $\bP^1$ that are not hyperbolic}
\label{table:cycP1}
\end{table}

We use the following notation in Table \ref{table:cycP1}.
It is clear that $\#J\leq 4$ if $\chitop(C^\nu)\geq 0$. Let $\{x\}$ be
the fractional part of a real number $x$. Let $p<q$ be two positive integers
that are relatively prime. Denote by $C_1^\nu$ a connected component of $C^\nu$.
Let $\gamma:=\gcd(d,a_1,\cdots,a_l)$ be the number of connected 
components of $C^\nu$. The fact that $\gamma$ is always
an integer gives certain divisibility condition on $d$.
\end{proof}

Next, we will discuss about the hyperbolicity of cyclic covers of $\bA^1$. 

Let us identify $\bA^1$ with $\bP^1\setminus\{\infty\}$.
Denote $C^0:=f^{-1}(\bA^1)$. If $D$ is supported away from $\infty$, then
$C^0$ is the degree $d$ cyclic cover of $\bA^1$ branched along an effective divisor $D$ of degree $m$. Since $\infty\in D$, we have $\# (f^\nu)^{-1}(\infty)=d$. Hence
\begin{align*}
\chitop((C^0)^\nu)& = \chitop(C^\nu)-d \\
& =d-\sum_{i=1}^l (d-\gcd(a_i,d))
\end{align*}

It is easy to see that $(C^0)^\nu$ is a 
disjoint union of $\gcd(d, a_1,\cdots,a_l)$ isomorphic copies of an irreducible 
 smooth affine curve. Therefore, $C^0$ is Brody hyperbolic if and only if $\chitop((C^0)^\nu)<0$.

\begin{lem}\label{fiberhyp2}
With the above notation, 
assume in addition that $D$ is supported away from $\infty$. 
Moreover, assume that one of the following holds:
\begin{enumerate}
 \item[-] $d$ is even and $l\geq \frac{m}{d}+2$;
 \item[-] $d$ is odd and $l\geq \frac{m}{d}+1$.
\end{enumerate}
Then $\chitop((C^0)^\nu)<0$, i.e. $C^0$ is Brody hyperbolic.
\end{lem}

\begin{proof}
Assume to the contrary that $\chitop((C^0)^\nu)\geq 0$.
Define an index set $J:=\{j\mid 1\leq j\leq l,\, a_j \textrm{ is not a multiple of }d\}$.
Notice that $\gcd(a_i,d)\leq d/2$ if $a_i$ is not a multiple of $d$.
So we have $\chitop((C^0)^\nu)<0$ as soon as $\#J\geq 3$.
Hence we only need to consider cases when $\#J\leq 2$. For simplicity, 
we may assume that $J=\{1,\cdots,\#J\}$.

After a careful study we get the following table, which 
illustrates all cases when $\chitop((C^0)^\nu)\geq 0$, i.e. 
when $C^0$ is not Brody hyperbolic, up to permutations of $a_1,\cdots, a_l$.
\begin{table}[H]
\[\begin{array}{|c|c|c|c|c|c|}
\hline \#J & \{a_1/d\} & \{a_2/d\} & (C^0)_1^\nu & \gamma & l\leq\\
\hline 0 & 0 & 0 & \bA^1 & d & m/d\\
\hline 2 & 1/2 & 1/2 & \bA^1\setminus\{0\} & d/2 & m/d+1\\ 
\hline
\end{array}
\]
\caption{Cyclic covers of $\bA^1$ that are not hyperbolic}
\label{table:cycA1}
\end{table}
 
We use the following notation for Table \ref{table:cycA1}. It is clear that $\#J\leq 2$ if $\chitop((C^0)^\nu)\geq 0$. 
Denote by $(C^0)_1^\nu$ a connected component of $(C^0)^\nu$.
Let $\gamma:=\gcd(d,a_1,\cdots,a_l)$ be the number of
connected components of $(C^0)^\nu$. The fact that $\gamma$ is always
an integer gives certain divisibility condition on $d$.
\end{proof}

\subsection{Proofs}\label{pf}
To begin with, we will study the enumerative geometry problem of counting the intersections of a generic
hypersurface with the generators of the projective cone $\tphi(E)$ (see also \cite[1.3]{zai09}).

A map $\alpha:\bP^1\to E$ is called a 
\textit{ruling} if it parametrizes a fiber of the $\bP^1$-bundle projection $\rho_E:E\to H$.
We say that $\alpha$ {\it corresponds to} $x\in H$
if $\alpha$ parametrizes $\rho_{E}^{-1}(x)$.

\begin{lem}\label{tangentdim}
With the above notation, the following properties hold for a general hypersurface
 $T\in |\mathcal{O}_{\bP^N}(m)|$:
 \begin{enumerate}
  \item $\phi(v_0)$ is not contained in $T$;
  \item For any ruling $\alpha:\bP^1\to E$, $(\tphi
  \circ\alpha)^*(T)$ is supported at $>(m-n)$ points.
 \end{enumerate} 
\end{lem}

\begin{proof}
Define $|\mathcal{O}_{\bP^N}(m)|^\circ:=\{T\in|\mathcal{O}_{\bP^N}(m)|:
\phi(v_0)\textrm{ is not contained in }T\}$. Denote by
 $\alpha_x:\bP^1\rightarrow E$ the ruling of $E$
 corresponding to $x\in H$.

We define an incidence variety $Z$ as 
\[
Z:=\{(T,x)\in |\mathcal{O}_{\bP^N}(m)|^\circ\times H: 
(\tphi\circ\alpha_x)^*(T)\textrm{ is supported at }
\leq (m-n)\textrm{ points} \}.
\]
Denote the two projections from $Z$ by $pr_1$ and $pr_2$.
Let $Z_x$ be the fiber of $pr_2:Z\to H$ over $x$. 
Then 
\[
Z_x\cong\{ T\in |\mathcal{O}_{\bP^N}(m)|^\circ: 
(\tphi\circ\alpha_x)^*(T)\textrm{ is supported at }
\leq (m-n)\textrm{ points} \}.
\]
By Proposition \ref{tcX} (\ref{tcX3}), $(\tphi\circ\alpha_x)$ parametrizes a line in $\bP^N$. Therefore, the rational map
\[
(\tphi\circ\alpha_x)^*: |\mathcal{O}_{\bP^N}(m)|\dashrightarrow |\mathcal{O}_{\bP^1}(m)|
\]
is a projection between projective spaces.
In particular, $(\tphi\circ\alpha_x)^*: |\mathcal{O}_{\bP^N}(m)|^\circ\to |\mathcal{O}_{\bP^1}(m)|$ is a
flat morphism for any $x\in H$. 

Let $W_k:=\{D\in|\mathcal{O}_{\bP^1}(m)|: D\textrm{ is supported at }\leq k \textrm{ points}\}$.
Then $\dim W_k=k$. It is clear that $Z_x=((\tphi\circ\alpha_x)^*)^{-1}(W_{m-n})$, so
\begin{align*}
\dim Z_x & =\dim ((\tphi\circ\alpha_x)^*)^{-1}(W_{m-n})\\
& =\dim |\mathcal{O}_{\bP^N}(m)|^\circ - \dim |\mathcal{O}_{\bP^1}(m)| + \dim W_{m-n}\\
& =\dim |\mathcal{O}_{\bP^N}(m)|^\circ - n
\end{align*} 
Hence 
$\dim Z=\dim Z_x+\dim H = \dim |\mathcal{O}_{\bP^N}(m)|^\circ - 1$ for a general choice of $x$, which implies that
$\dim pr_1(Z)\leq \dim |\mathcal{O}_{\bP^N}(m)|^\circ - 1$.
Thus the map $pr_1$ is not surjective, which means that a general
hypersurface $T\in |\mathcal{O}_{\bP^N}(m)|^\circ$ will satisfy 
property (2).
The lemma then follows automatically.
\end{proof}

Since property \ref{star} holds for a general hypersurface 
$T\in|\mathcal{O}_{\bP^N}(m)|$, Lemma \ref{tangentdim} implies that
the following property also holds for general $T$.

\begin{enumerate}
 \item[($**$)\namedlabel{doublestar}{($**$)}] \qquad$\phi^*(T)$ is smooth, does not contain $v_0$
 and intersects $X_0$ transversally.
 Besides, $(\tphi\circ\alpha_x)^*(T)$ is supported
 at $\geq(m-n+1)$ points for any $x\in H$.
\end{enumerate}

\medskip

From now on we always fix a general hypersurface 
$T\in |\mathcal{O}_{\bP^N}(m)|$. Then we may assume that
$T$ satisfies \ref{doublestar}.
\medskip

The following theorem is the main tool to prove Theorem \ref{mainthm}.
\begin{thm}[{\cite[3.11.1]{kob98}}]
\label{deformation}
Let $\pi:\cX\rightarrow R$ be a proper family of connected complex analytic spaces. If there is a point $r_0\in R$ such that the fiber
$X_{r_0}$ is Brody hyperbolic, then there exists an open neighborhood 
(in the Euclidean topology)
$U\subset R$ of $r_0$ such that for each $r\in U$,
the fiber $X_r$ is Brody hyperbolic.
\end{thm}

We will prove the following theorem, a stronger result that implies Theorem \ref{mainthm}.
\begin{thm}\label{refined1}
Let $X$ be a smooth projective variety with $\dim X=n$. Let $L$ be a globally
 generated ample line bundle on $X$.
 Suppose that there exists a smooth hypersurface $H\in |L|$ that is Brody hyperbolic.
 Let $m,d\geq 2$ be positive integers such that $m$ is a multiple of $d$. 
 For a generic small deformation $S$ of
 $mH\in |L^{\otimes m}|$, let $Y$ be the degree $d$ cyclic cover of $X$ branched along $S$. 
 Then $Y$ is Brody hyperbolic if one of the following holds:
\begin{enumerate}
 \item[-] $d$ is divisible by $2$ or $3$, and $m\geq d\lceil \frac{n+2}{d-1} \rceil$;
 \item[-] $d$ is relatively prime to 6, and $m\geq d\lceil \frac{n+1}{d-1} \rceil$.
 \end{enumerate}
\end{thm}

\begin{proof}\label{proof1}
We first show that $Y_0$ is Brody hyperbolic.

Proposition \ref{tcY} implies that the birational morphism
$\psi_1|_{\oY_0}:\oY_0\to Y_0$ induces an isomorphism between $\oY_0\setminus(\hX_{0,1}\cup\cdots\cup
\hX_{0,d})$ and $Y_0\setminus\{q_1,\cdots,q_d\}$. Therefore, it suffices
to show that $\oY_0$ is Brody hyperbolic.
 
Define $p_{\oY_0}:\oY_0\to H$ to be the composition map $\oY_0\to E\to H$. 
Since $H$ is Brody hyperbolic, we only need to show that
every fiber $p_{\oY_0}^{-1}(x)$ is Brody hyperbolic. 
It is clear that 
$p_{\oY_0}^{-1}(x)$ is the degree $d$ cyclic cover of 
$\rho_E^{-1}(x)$ branched along 
$\tphi^*(T)|_{\rho_E^{-1}(x)}$. Applying 
the pull back of a ruling $\alpha_x$ yields that  
$p_{\oY_0}^{-1}(x)\cong C_x$, where $C_x$ is the degree $d$ cyclic cover
of $\bP^1$ branched along $(\tphi\circ \alpha_x)^*(T)$.

Let $l_x:=\#\Supp\left((\tphi\circ \alpha_x)^*(T)\right)$.
Since $T$ satisfies \ref{doublestar}, $l_x\geq m-n+1$
for any $x\in H$. 
If $d$ is divisible by $2$ or $3$, then $m\geq d\lceil \frac{n+2}{d-1} \rceil\geq \frac{d}{d-1}(n+2)$.
Hence $l_x\geq m-n+1\geq\frac{m}{d}+3$.
If $d$ is relatively prime to $6$, then $m\geq d\lceil \frac{n+1}{d-1} \rceil\geq \frac{d}{d-1}(n+1)$. 
Hence $l_x\geq m-n+1\geq\frac{m}{d}+2$. 
Then Lemma \ref{fiberhyp1} implies that $C_x$ 
is Brody hyperbolic. 

Summing up, we always have that $p_{\oY_0}^{-1}(x)\cong C_x$ 
is Brody hyperbolic for any $x\in H$.
Therefore, $Y_0$ is Brody hyperbolic. 

We may apply Theorem 
\ref{deformation} to the family $\cY\rightarrow \bA^1$ with
$r_0=0$. Thus $Y_t$ is Brody hyperbolic for $|t|$ sufficiently
small. By Proposition \ref{tcY}, $Y_t$ is isomorphic
to the degree $d$ cyclic cover of $X$ branched along $S_t$, where $S_t$
is a generic small deformation of $mH$ for $|t|$ sufficiently small.
The theorem then follows.
\end{proof}

Next we give another deformation type theorem of hyperbolicity when the special fiber has multiple irreducible components.
It will be used to prove Theorem \ref{mainthm2}. Note that some
cases of Theorem \ref{sz} have already been used in 
\cite{sz05,zai09}.

\begin{thm}\label{sz}
Let $\pi:\cX\rightarrow R$ be a proper family of connected complex analytic spaces over a nonsingular complex curve $R$. Let $r_0\in R$ be a point.
Denote the irreducible components of the fiber $X_{r_0}$ by $X_{r_0,1},\cdots, X_{r_0,k}$.
Suppose these data satisfy the following properties:

\begin{enumerate}
\item $X_{r_0,i}$ is a Cartier divisor on $\cX$ for each $1\leq i\leq k$;
\item For any partition of indices $I\cup J=\{1,\cdots,k\}$, 
$\bigcap_{i\in I} X_{r_0,i}\setminus \bigcup_{j\in J} X_{r_0,j}$ 
is Brody hyperbolic.
\end{enumerate}

Then there exists an open neighborhood (in the Euclidean
topology)
$U\subset R$ of $r_0$ such that for each $r\in U\setminus\{r_0\}$,
the fiber $X_r$ is Brody hyperbolic.
\end{thm}

\begin{proof}
Assume to the contrary that there exists a sequence of points $\{r_n\}$ converging to $r_0$ such that $X_{r_n}$ is not Brody hyperbolic for each $n$. Then there is a complex line $h_n:\bC\to X_{r_n}$. By taking a subsequence of $\{r_n\}$ if necessary, we may assume that $\{h_n\}$ converges to a complex line $h:\bC\to X_{r_0}$. Then by applying the generalized Hurwitz theorem \cite[3.6.11]{kob98} to $(\cX,X_{r_0})$, we have that 
\[
h(\bC)\subset\bigcap_{i\in I} X_{r_0,i}\setminus \bigcup_{j\in J} X_{r_0,j},
\]
where $I=\{i: h(0)\in X_{r_0,i}\}$ and $J=\{j: h(0)\not\in X_{r_0,j}\}$. However, $\bigcap_{i\in I} X_{r_0,i}\setminus \bigcup_{j\in J} X_{r_0,j}$ is Brody hyperbolic, we get a contradiction!
\end{proof}

The following theorem is a stronger result that implies Theorem \ref{mainthm2}.

\begin{thm}\label{refined2}
With the notation of Theorem \ref{refined1}, assume in addition that
 $X\setminus H$ is Brody hyperbolic. Then $Y$ is Brody
hyperbolic if one of the following holds:
\begin{enumerate}
 \item[-] $d$ is even and $m\geq d\lceil\frac{n+1}{d-1}\rceil$;
 \item[-] $d$ is odd and $m\geq d\lceil\frac{n}{d-1}\rceil$.
\end{enumerate}
\end{thm}

\begin{proof}\label{proof2}
We first prove that $\oY_0\setminus(\hX_{0,1}\cup\cdots\cup
\hX_{0,d})$ is Brody hyperbolic.

Consider the restriction of $p_{\oY_0}:\oY_0\to H$ on the
open subset $\oY_0\setminus(\hX_{0,1}\cup\cdots\cup
\hX_{0,d})$. Since $H$ is 
Brody hyperbolic, we only need to show that 
$p_{\oY_0}^{-1}(x)\setminus(\hX_{0,1}\cup\cdots\cup
\hX_{0,d})$ is Brody 
hyperbolic for any $x\in H$.

It is clear that $\oY_0$ is the 
degree $d$ cyclic cover of $E$ branched along 
$\tphi^*(T)|_E$, and $\hX_{0,1}\cup\cdots\cup
\hX_{0,d}=\tmu^{-1}(\hX_0)$. Therefore, 
$p_{\oY_0}^{-1}(x)\setminus(\hX_{0,1}\cup\cdots\cup
\hX_{0,d})$ is the preimage
of $\rho_E^{-1}(x)\setminus\hX_0$ under
the covering map $p_{\oY_0}^{-1}(x)\to 
\rho_E^{-1}(x)$. Applying the pull back of a ruling $\alpha_x$
with $\{\alpha_x(\infty)\}=\rho_E^{-1}(x)\cap\hX_0$
yields that $p_{\oY_0}^{-1}(x)\setminus(\hX_{0,1}\cup\cdots\cup
\hX_{0,d})\cong C_x^0$ (with the notation of Lemma \ref{fiberhyp2}
and Proof \ref{proof1}, because
$(\tphi\circ \alpha_x)^*(T)$ is supported
away from $\infty$). 

Since $T$ satisfies \ref{doublestar}, $l_x\geq m-n+1$ for any $x\in H$.
If $d$ is even, then $m\geq d\lceil\frac{n+1}{d-1}\rceil\geq 
\frac{d}{d-1}(n+1)$. Hence $l_x\geq m-n+1\geq\frac{m}{d}+2$.
If $d$ is odd, then $m\geq d\lceil\frac{n}{d-1}\rceil\geq 
\frac{d}{d-1}\cdot n$. Hence $l_x\geq m-n+1\geq\frac{m}{d}+1$.
Then Lemma \ref{fiberhyp2} implies that 
$C_x^0$ is Brody hyperbolic. 

Summing up, we always have that $p_{\oY_0}^{-1}(x)\setminus(\hX_{0,1}\cup\cdots\cup
\hX_{0,d})\cong C_x^0$ is Brody hyperbolic for any $x\in H$. Therefore,
$\oY_0\setminus(\hX_{0,1}\cup\cdots\cup
\hX_{0,d})$ is Brody hyperbolic.

On the other hand, $\hX_{0,i}\setminus\oY_0$ is isomorphic to 
$\hX_0\setminus E$, which is again isomorphic
to $X\setminus H$. Hence $\hX_{0,i}\setminus\oY_0$ is Brody hyperbolic 
for any $1\leq i\leq d$. We may apply Theorem \ref{sz} to the
family $\tcY\rightarrow \bA^1$ with $r_0=0$.
Thus $Y_t$ is Brody hyperbolic for $|t|$ sufficiently
small. By Proposition \ref{tcY}, $Y_t$ is isomorphic
to the degree $d$ cyclic cover of $X$ branched along $S_t$, where $S_t$
is a generic small deformation of $mH$ for $|t|$ sufficiently small.
The theorem then follows.
\end{proof}

Finally, we apply our arguments to hyperbolicity of the complements.
We first state a theorem which relates various notions of hyperbolicity
of complements.
\begin{thm}[Green \cite{gre77}; Howard \cite{howard}]\label{greenhoward}
Let $X$ be a compact complex space. Let $S$ be an effective Cartier divisor on $X$. If both $S$ and $X\setminus S$ are Brody hyperbolic, then $X\setminus S$ is complete hyperbolic and hyperbolically embedded in $X$.
\end{thm}

Next we give a deformation type theorem of hyperbolicity of complements of 
Cartier divisors, which will be used
to prove Theorem \ref{complement}.

\begin{thm}\label{def-comp}
Let $\pi:\cX\rightarrow R$ be a proper family of connected complex 
analytic spaces. Let $\mathcal{S}$ be an effective Cartier divisor on
$\cX$. Assume that there is a point $r_0\in R$ satisfying the following
properties:
\begin{enumerate}
\item Both $S_{r_0}$ and
$X_{r_0}\setminus S_{r_0}$ are Brody hyperbolic.
\item $\mathcal{S}$ does not contain any irreducible component of $X_{r_0}$.
\end{enumerate}
Then there exists an open
neighborhood (in the Euclidean
topology) $U\subset R$ of $r_0$ such that for each $r\in U$,
both $S_r$ and $X_r\setminus S_r$ are Brody hyperbolic. 
\end{thm}

\begin{proof}
 Let $U_r:=X_r\setminus S_r$. Since $S_{r_0}$ is Brody hyperbolic, Theorem \ref{deformation} implies that $S_r$ is Brody hyperbolic for $r$ sufficiently close to $r_0$. Therefore,
 it suffices to show that $U_r$ is Brody hyperbolic for $r$ in a small neighborhood of $r_0$.
 
 Assume to the contrary that there exists a sequence of points $\{r_n\}$ converging to $r_0$ such that
 $U_{r_n}$ is not Brody hyperbolic for each $n$. Hence $U_{r_n}$ is not hyperbolically embedded in $X_{r_n}$. Then there is a limit complex line $h_n:\bC\to X_{r_n}$ coming from $U_n$. By taking a subsequence of $\{r_n\}$ if necessary,
 we may assume that $\{h_n\}$ converges to a complex line $h:\bC\to X_{r_0}$. Then by applying the generalized Hurwitz theorem \cite[3.6.11]{kob98} to $(\cX,\mathcal{S})$, $h(\bC)$ is either contained in $S_{r_0}$ or $U_{r_0}$. However, both $S_{r_0}$ and $U_{r_0}$ are
 Brody hyperbolic, we get a contradiction!
\end{proof}

\begin{emp}[\it Proof of Theorem \ref{complement}]\label{proof3}
We first prove that $X_0 \setminus \phi^*(T)$
is Brody hyperbolic.

It is clear that the birational morphism
$\psi|_{E}:E\to X_0$ induces an isomorphism between $E\setminus \hX_0$ and $X_0\setminus\{v_0\}$. Therefore, it suffices
to show that $E\setminus \tphi^*(T)$ is Brody 
hyperbolic.

Consider the restriction of $\rho_E:E\to H$ on the open
subset $E\setminus \tphi^*(T)$. Since $H$ is Brody
hyperbolic, we only need to show that $\rho_E^{-1}(x)
\setminus \tphi^*(T)$ is Brody hyperbolic. Applying the pull back
of a ruling $\alpha_x$ yields that
$\rho_E^{-1}(x)\setminus \tphi^*(T)\cong 
  \bP^1\setminus (\tphi\circ\alpha_x)^*(T)$.

Since $T$ satisfies \ref{doublestar}, $l_x\geq m-n+1$ for any $x\in H$. Then the assumption $m\geq n+2$ implies that 
$l_x\geq 3$. Hence $\bP^1\setminus (\tphi\circ
\alpha_x)^*(T)$ is Brody hyperbolic, which means that $\rho_E^{-1}(x)\setminus \tphi^*(T)$ is also Brody hyperbolic. Consequently,
$X_0 \setminus \phi^*(T)$ is Brody hyperbolic.

On the other hand, since $\tphi^*(T)$ is disjoint from $\hX_0$, 
no fiber of $\rho_E$ is contained in $\tphi^*(T)|_E$. Hence
the restriction of $\rho_E$ on $\tphi^*(T)|_E$
is a finite morphism onto $H$. Then $H$ being Brody hyperbolic implies 
that $\tphi^*(T)|_E$ is Brody hyperbolic.
Recall that $\phi^*(T)|_{X_0}\cong\tphi^*
(T)|_E$, so $\phi^*(T)|_{X_0}$ is also Brody 
hyperbolic.

So far we have shown that both $\phi^*(T)|_{X_0}$ and
$X_0 \setminus \phi^*(T)$ are Brody hyperbolic.
Applying Theorem \ref{def-comp} to the family 
$\cX\rightarrow \bA^1$ with $\mathcal{S}=\phi^*(T)$ and $r_0=0$ yields that both 
$\phi^*(T)|_{X_t}$ and $X_t\setminus \phi^*(T)$
are Brody hyperbolic for $|t|$ sufficiently small. Since 
$(X_t,\phi^*(T)|_{X_t})$ is isomorphic to $(X,S_t)$, both
$S_t$ and $X\setminus S_t$ are Brody hyperbolic for $|t|$
sufficiently small. By Proposition \ref{tcY}, $S_t$ is a generic small deformation of $mH$ for $|t|$ sufficiently small.
The first statement of the theorem then follows. The last statement follows directly from Theorem \ref{greenhoward}. 
\qed
\end{emp}

\section{Applications and Examples}\label{ex}

\subsection{Hypersurfaces in \texorpdfstring{$\bP^n$}{Pn}}\label{ex-hyp}\hfill\par
Let us introduce some notation to describe the moduli spaces of hypersurfaces in $\bP^n$ with various hyperbolic conditions.
\begin{itemize}
\item Let $\bP_{n,\delta}$ be the projective space of dimension 
$\binom{n+\delta}{n}-1$ whose points parametrizes hypersurfaces of
degree $\delta$ in $\bP^n$.
\item Let $H_{n,\delta}\subset\bP_{n,\delta}$ be
the subset corresponding to Brody hyperbolic hypersurfaces.
\item Denote by 
$HE_{n,\delta}$ the subset of $\bP_{n,\delta}$ consisting of the
hypersurfaces of degree $\delta$ in $\bP^n$
with hyperbolically embedded complements. 
\end{itemize}

The following theorem (which essentially implies Theorem \ref{cyclichyp}) produces Brody hyperbolic hypersurfaces that are cyclic covers of $\bP^n$.

\begin{thm}\label{hcc}
Let $k,\delta$ be positive integers such that $\delta$ is a multiple of $k$. Suppose one of the following conditions holds.
\begin{enumerate}
\item $H_{n,k}$ is non-empty and $\delta\geq (n+3)k$;
\item $HE_{n,k}\cap H_{n,k}$ is non-empty and $\delta\geq (n+2)k$.
\end{enumerate}
Then there exists a Brody hyperbolic smooth hypersurface $W$ of degree $\delta$ in $\bP^{n+1}$, such that $W$ is a cyclic cover of $\bP^n$ under some linear projection.
\end{thm}

\begin{proof}\hfill
\par (1) Choose a smooth hypersurface $D\in H_{n,k}$. Let $d:=\delta/k$. For any $d\geq n+3$, apply Theorem \ref{mainthm} to $(X,L,H,d,m):=(\bP^n, \cO(k), D, d, d)$ yields that there exists a degree $d$ cyclic cover $Y$ of $\bP^n$ branched along a smooth hypersurface $S$ of degree $\delta$ such that $Y$ is Brody hyperbolic. Let $W$ be the degree $\delta$ cyclic cover of $\bP^n$ branched along $S$. Then $W\to Y$ is a finite surjective morphism. Thus $Y$ being Brody hyperbolic implies that $W$ is also Brody hyperbolic.

(2) Choose a smooth hypersurface $D\in HE_{n,k}\cap H_{n,k}$. Let $d:=\delta/k$. For any $d\geq n+2$, apply Theorem \ref{mainthm2} to $(X,L,H,d,m):=(\bP^n, \cO(k), D, d, d)$. The rest of the proof is the same as (1).
\end{proof}

Next, we give a new proof to \cite[1.1]{zai09} using Mori's degeneration method.

\begin{thm}[Zaidenberg \cite{zai09}]\label{zaicone}
 Let $X=(F(z_0,\cdots,z_n)=0)$ be a Brody hyperbolic hypersurface of degree $k$ in $\bP^n$ ($n\geq 2$). We may realize
 $\bP^n$ as the hyperplane $(z_{n+1}=0)$ in $\bP^{n+1}$. Denote by $C(X):=(F(z_0,\cdots,z_n)=0)\subset\bP^{n+1}$ the projective cone over $X$.
 Let $dC(X):=(F^d=0)\subset\bP^{n+1}$ be the $d$-th thickening of $C(X)$ where $d\geq 2$ is a positive integer.
 Then a generic small deformation of $dC(X)$ (in the sense of \cite{zai09}) is Brody hyperbolic. In particular, $H_{n,k}\neq\emptyset$ implies that $H_{n+1,dk}\neq\emptyset$ for $d\geq 2$.
\end{thm}

\begin{proof}
 Firstly, let us recall Mori's degeneration method from \cite{mor75}.
 
 Let $\bP(1^{n+2}, k)$ be a weighted projective space of dimension $(n+2)$ with coordinates $z_0,\cdots,z_{n+1},w$. Let $G$ be a
 general homogeneous polynomial of degree $dk$ in $z_0,\cdots,z_{n+1}$. Consider the family of complete intersections
 \[
 Y_t:=\left(tw-F(z_0,\cdots,z_n)=w^d-G(z_0,\cdots,z_{n+1})=0\right)\subset\bP(1^{n+2}, k).
 \]
  For $t\neq 0$ we can eliminate $w$ to obtain a degree $dk$ smooth hypersurface
 \[
 Y_t\cong\left(F^d(z_0,\cdots,z_n)=t^d G(z_0,\cdots,z_{n+1})\right)\subset\bP^{n+1}.
 \]
  For $t=0$ we see that $\cO_{Y_0}(1)$ is not very ample but realizes $Y_0$ as a degree $d$ cyclic cover
 \[
 h:Y_0\to C(X)=\left(F(z_0,\cdots,z_n)=0\right)\subset \bP^{n+1}
 \]
 of $C(X)$ branched along $(F=G=0)$. 
 
 Next, we will show that $Y_0$ is Brody hyperbolic. 
 Let us fix a general homogeneous polynomial $G$ from now on.
 By Lemma \ref{tangentdim}, a general hypersurface $T:=(G=0)$
 satisfies that $T$ does not contain the vertex $[0,\cdots,0,1]$ of $C(X)$, and that 
 \[
 \#(T\cap \ell)\geq dk-n
 \] for any generator $\ell$ of $C(X)$. Applying Lemma 
 \ref{fiberhyp1} to $(d,m,l):=(d,dk, \#(T\cap\ell))$ yields that
 $h^{-1}(\ell)$ is Brody hyperbolic if $\#(T\cap\ell)\geq k+3$.
 Since $X$ is a Brody hyperbolic hypersurface of degree $d$ in $\bP^n$ 
 with $n\geq 2$, it is clear that $k\geq n+3$. Hence 
 \[
 \#(T\cap\ell)\geq dk-n\geq k+3.
 \]
 Thus $h^{-1}(\ell)$ is Brody hyperbolic, which together with $X$ 
 being Brody hyperbolic implies that $Y_0$ is Brody hyperbolic.
 
 Finally, Theorem \ref{deformation} implies that $Y_t$ is Brody 
 hyperbolic for $|t|$ sufficiently small. Hence we prove the theorem.
\end{proof}

The following theorem is an improvement of 
\cite[p. 147, Corollary of Theorem II.2]{zai92}, where they assumed 
$\delta\geq (2n+1)k$ (without assuming $\delta$ being a multiple of
$k$).

\begin{thm}\label{hypersurface}
 If $H_{n,k}$ is non-empty, then $HE_{n,\delta}\cap H_{n,\delta}$ is a
 non-empty open subset of $\bP_{n,\delta}$ (in the Euclidean topology) 
 for any $\delta\geq (n+2)k$ with $\delta$ being a multiple of $k$.
\end{thm}

\begin{proof}
 Let $m=\delta/k\geq n+2$. Apply Theorem \ref{complement} to $X=\bP^n$, 
 $L=\cO_{\bP^n}(k)$, $H\in H_{n,k}$ yields that both $S$ and $\bP^n
 \setminus S$ are Brody hyperbolic for a generic small 
 deformation $S$ of $mH\in|\cO_{\bP^n}(\delta)|$. Theorem \ref{greenhoward}
 implies that $S\in HE_{n,\delta}\cap H_{n,\delta}$ for any generic small 
 deformation $S$ of $mH$, hence $HE_{n,\delta}\cap H_{n,\delta}$ is
 non-empty. The openness of $HE_{n,\delta}\cap H_{n,\delta}$ follows
 from Theorem \ref{greenhoward} and \ref{def-comp}.
\end{proof}

\subsection{Surfaces}\label{ex-surface}
\hfill

The following theorem provides new examples of hyperbolic surfaces in $\bP^3$ of minimal degree $15$.

\begin{thm}
Let $\delta=d\cdot k$ be the product of two positive integers $d\geq 3$, $k\geq 5$.
Then there exists a smooth Brody hyperbolic surface $X_\delta$ of degree $\delta$ in $\bP^3$ that is a cyclic cover of $\bP^2$ under some linear projection.
\end{thm}

\begin{proof}
According to \cite{zai89}, for any $k\geq 5$ there exists a smooth curve $D$ in $\bP^2$ of degree $k$ with $\bP^2\setminus D$ being Brody hyperbolic. Then the proof is along the same line as of Theorem \ref{hcc} (2), except that we apply Theorem \ref{refined2} instead of Theorem \ref{mainthm2} when $d=3$.
\end{proof}

Next, we prove Theorem \ref{ccp2}.

\begin{thm}[=Theorem \ref{ccp2}]\label{cyclicp2}
Let $l\geq 3$, $k\geq 5$ be two positive integers. Let $D$ be a
smooth plane curve of degree $k$ such that $\bP^2\setminus D$ 
is Brody hyperbolic. (The existence of such $D$ was shown by 
Zaidenberg in \cite{zai89}.) Let $S$ be a generic small deformation 
of $2l D$. Then the double cover of $\bP^2$ branched along $S$ 
is Brody hyperbolic.
\end{thm}

\begin{proof}
 Apply Theorem \ref{mainthm2} to $(X,L,H,d,m):=(\bP^2, \cO(k), D, 2, 2l)$.
\end{proof}

\begin{expl}\label{k3}
Let $(X_0,L_0)$ be a primitively polarized K3 surface of degree $2l$ for $l\in\bZ_{>0}$. For any $m\geq 4$, denote by $M_0:=L_0^m$. Pick a general member $H\in |L_0|$, then $H$ is smooth and $g(H)\geq 2$. Let $S_0$ be a generic small deformation of $mH$ that is smooth. Then Theorem \ref{complement} implies that $S_0$ and $X_0\setminus S_0$ are both Brody hyperbolic.

There exists a deformation $(\cX,\cM)$ of $(X_0, M_0)$ over $\Delta$ such that $(X_t,M_t)$ is a primitively polarized K3 surface of degree $2lm^2$ for $t\in\Delta\setminus\{0\}$. It is clear that $h^0(X_t,M_t)$ does not depend on the choice of $t$ in $\Delta$. Hence Grauert's theorem implies that $\pi_*\cM$ is a locally free sheaf on $\Delta$, where $\pi:\cX\to\Delta$ is the projection map. In other
words, $\{H^0(X_t,M_t)\}_{t\in\Delta}$ forms a holomorphic vector bundle over $\Delta$. Now we may deform $S_0$ to a family of divisors $S_t\in |M_t|$ for $|t|$ sufficiently small. By choosing a generic deformation, we may assume that $S_t$ is smooth for $|t|$ sufficiently small. Hence Theorem \ref{def-comp} implies that both $S_t$ and $X_t\setminus S_t$ are Brody hyperbolic for $|t|$ sufficiently small. Moreover, $X_t\setminus S_t$ is complete hyperbolic and hyperbolically embedded in $X_t$ by Theorem \ref{greenhoward}.

As a consequence, for any $l\geq 1$ and $m\geq 4$ there exists a primitively polarized K3 surface $(X,M)$ of degree $2lm^2$ 
and a smooth curve $S\in |M|$, such that $X\setminus S$ is complete hyperbolic and hyperbolically embedded in $X$. Notice that the minimal degree of $(X,M)$ is $32$.
\end{expl}


\begin{thebibliography}{99}


 \bibitem[Bro16]{bro16} Damian Brotbek: {\it On the hyperbolicity of general 
 hypersurfaces}.  Preprint available at \href{http://arxiv.org/abs/1604.00311}
 {\textsf{arXiv:1604.00311}}.

 \bibitem[Cle86]{cle86} Herbert Clemens: {\it Curves on generic hypersurfaces}.
 Ann. Sci. \'Ecole Norm. Sup. (4) 19 (1986), no. 4, 629-636.

 \bibitem[CZ03]{cz03} Ciro Ciliberto and Mikhail Zaidenberg: {\it 
 $3$-fold symmetric products of curves as hyperbolic hypersurfaces
 in $\bP^4$}. Internat. J. Math. 14 (2003), no. 4, 413-436.

 \bibitem[CZ13]{cz13} Ciro Ciliberto and Mikhail Zaidenberg: {\it Scrolls and
 hyperbolicity}. Internat. J. Math. 24 (2013), no. 4, 1350026, 25 pp. 
 
 \bibitem[Dem15]{dem15} Jean-Pierre Demailly: {\it 
 Proof of the Kobayashi conjecture on the hyperbolicity of very 
 general hypersurfaces}.
 Preprint available at \href{http://arxiv.org/abs/1501.07625}
 {\textsf{arXiv:1501.07625}}.
 
 \bibitem[DEG00]{deg00} Jean-Pierre Demailly and Jawher El Goul:
 {\it Hyperbolicity of generic surfaces of high degree in projective $3$-space}.
 Amer. J. Math. 122 (2000), no. 3, 515-546.
 
 \bibitem[Den16]{den16} Ya Deng: {\it Effectivity in the 
 hyperbolicity-related problems
}. Preprint available at \href{http://arxiv.org/abs/1606.03831}
 {\textsf{arXiv:1606.03831}}.
 
 \bibitem[DMR10]{dmr10} Simone Diverio, Joel Merker and Erwan Rousseau:
 {\it Effective algebraic degeneracy}. Invent. Math. 180 (2010), no. 1,
 161-223.
 
 \bibitem[DT10]{dt10} Simone Diverio and Stefano Trapani: {\it A remark on the codimension of the Green-Griffiths locus of generic projective hypersurfaces of high degree}. 
J. Reine Angew. Math. 649 (2010), 55-61.
 
 \bibitem[Duv04]{duv04} Julien Duval: {\it Une sextique hyperbolique dans $\bP^3
 (\bC)$}. 
 Math. Ann. 330 (2004), no. 3, 473-476. 
 
 \bibitem[Ein88]{ein88} Lawrence Ein: {\it Subvarieties of generic complete intersections}. Invent. Math. 94 (1988), no. 1, 163-169. 
 
 \bibitem[Ein91]{ein91} Lawrence Ein: {\it Subvarieties of generic 
 complete intersections. II}. Math. Ann. 289 (1991), no. 3, 465-471.
 
 \bibitem[EG03]{eg03} Jawher El Goul: {\it Logarithmic jets and hyperbolicity}.
 Osaka J. Math. 40 (2003), no. 2, 469-491.
 
 \bibitem[Fuj01]{fuj01} Hirotaka Fujimoto: {\it A family of hyperbolic hypersurfaces in the complex projective space}. The Chuang special issue. Complex Variables Theory Appl. 43 (2001), no. 3-4, 273-283.

 \bibitem[Gre77]{gre77} Mark Green: {\it The hyperbolicity of the
 complement of $2n+1$ hyperplanes in general position in $\bP^n$ and 
 related results}. Proc. Amer. Math. Soc. 66 (1977), no. 1, 109-113. 
 
 \bibitem[GG79]{GG} Mark Green and Phillip Griffiths: {\it Two applications
 of algebraic geometry to entire holomorphic mappings}. 
 The Chern Symposium 1979 (Proc. Internat. Sympos., Berkeley, Calif., 1979), 
 pp. 41-74, Springer, New York-Berlin, 1980. 
 
 \bibitem[How]{howard} Alan Howard. Unpublished manuscript.

  \bibitem[Huy15]{huy15} Dinh Tuan Huynh: {\it 
 Examples of hyperbolic hypersurfaces of low degree in projective spaces}.
 Preprint available at \href{http://arxiv.org/abs/1507.03542}
 {\textsf{arXiv:1507.03542}}.
 
  \bibitem[Huy16]{huy16} Dinh Tuan Huynh: {\it Construction of hyperbolic hypersurfaces of low degree in $\bP^n(\bC)$}. Preprint available at \href{http://arxiv.org/abs/1601.06653}
 {\textsf{arXiv:1601.06653}}.
 
  \bibitem[IT15]{it15} Atsushi Ito and Yusaku Tiba: {\it
Curves in quadric and cubic surfaces whose complements are Kobayashi hyperbolically imbedded}.
Annales de l'institut Fourier, 65 no. 5 (2015), p. 2057-2068.
 
  \bibitem[Kob70]{kob70} Shoshichi Kobayashi: {\it Hyperbolic manifolds and
  holomorphic mappings}. Pure and Applied Mathematics, 2 Marcel Dekker,
  Inc., New York 1970 ix+148 pp.
 
 \bibitem[Kob98]{kob98} Shoshichi Kobayashi: {\it
 Hyperbolic complex spaces}. Grundlehren der Mathematischen 
 Wissenschaften [Fundamental Principles of Mathematical Sciences],
 318. Springer-Verlag, Berlin, 1998. xiv+471 pp.
 
 \bibitem[Kol13]{kol13} J\'anos Koll\'ar: {\it
 Singularities of the minimal model program}. With a collaboration of S\'andor Kov\'acs.
 Cambridge Tracts in Mathematics, 200. Cambridge University Press, Cambridge, 2013. 
 x+370 pp.
 
 \bibitem[Lan86]{lang} Serge Lang: {\it Hyperbolic and Diophantine analysis}.
 Bull. Amer. Math. Soc. (N.S.) 14 (1986), no. 2, 159-205.
 
 \bibitem[Laz04]{laz04} Robert Lazarsfeld: {\it Positivity in algebraic geometry. I. Classical setting: line bundles and linear series}. Ergebnisse der Mathematik und ihrer Grenzgebiete. 3. Folge. A Series of Modern Surveys in Mathematics [Results in Mathematics and Related Areas. 3rd Series. A Series of Modern Surveys in Mathematics], 48. Springer-Verlag, Berlin, 2004. xviii+387 pp. 

 \bibitem[MN96]{mn96} Kazuo Masuda and Junjiro Noguchi: {\it A construction of 
 hyperbolic hypersurface of $\bP^n(\bC)$}. Math. Ann. 304 (1996), no. 2, 339-362.

 
 \bibitem[McQ99]{mcq99} Michael McQuillan: {\it Holomorphic curves on hyperplane sections of $3$-folds}.
 Geom. Funct. Anal. 9 (1999), no. 2, 370-392. 
 
 \bibitem[Mor75]{mor75} Shigefumi Mori: {\it On a generalization of complete intersections}. 
J. Math. Kyoto Univ. 15 (1975), no. 3, 619-646. 
 
 \bibitem[Pau08]{pau08} Mihai P\u{a}un: {\it Vector Fields on the total space of 
 hypersurfaces in the projective space and hyperbolicity}.
Math. Ann. 340 (2008), no. 4, 875-892. 

 \bibitem[RR13]{rr13} Xavier Roulleau and Erwan Rousseau: {\it On the
 hyperbolicity of surfaces of general type with small $c_1^2$}.
 J. Lond. Math. Soc. (2) 87 (2013), no. 2, 453-477.


 \bibitem[Rou07a]{rou07} Erwan Rousseau: {\it Weak analytic hyperbolicity of generic
 hypersurfaces of high degree in $\bP^4$}. 
  Ann. Fac. Sci. Toulouse Math. (6) 16 (2007), no. 2, 369-383.

 \bibitem[Rou07b]{rou07b} Erwan Rousseau: {\it Weak analytic hyperbolicity of complements of generic surfaces of high degree in projective $3$-space}. Osaka J. Math. 44 (2007), no. 4, 955-971.

 \bibitem[Rou09]{rou09} Erwan Rousseau: {\it Logarithmic vector fields and hyperbolicity}. Nagoya Math. J. 195 (2009), 21-40.

 \bibitem[SZ00]{sz00} Bernard Shiffman and Mikhail Zaidenberg: {\it Two classes of hyperbolic surfaces in $\bP^3$}. Internat. J. Math. 11 (2000), no. 1, 65-101.

 \bibitem[SZ02]{sz02} Bernard Shiffman and Mikhail Zaidenberg:
 {\it Hyperbolic hypersurfaces in $\bP^n$ of Fermat-Waring
type}. Proc. Amer. Math. Soc. 130 (2002), no. 7, 2031-2035.

 \bibitem[SZ05]{sz05} Bernard Shiffman and Mikhail Zaidenberg:
 {\it New examples of Kobayashi hyperbolic surfaces in $\bC\bP^3$}. (Russian) Funktsional. Anal. i Prilozhen. 39 (2005), no. 1, 90-94; translation in Funct. Anal. Appl. 39 (2005), no. 1, 76-79.

 \bibitem[Siu15]{siu15} Yum-Tong Siu: {\it Hyperbolicity of generic high-degree hypersurfaces
 in complex projective space}. Invent. Math. 202 (2015), no. 3, 1069-1166.

 \bibitem[SY96]{sy96} Yum-Tong Siu and Sai-Kee Yeung: {\it Hyperbolicity of the complement of a generic smooth curve of high degree in the complex projective plane}. Invent. Math. 124 (1996), no. 1-3, 573-618.

 \bibitem[SY97]{sy97} Yum-Tong Siu and Sai-Kee Yeung: {\it Defects for ample divisors of
 abelian varieties, Schwarz lemma,
and hyperbolic hypersurfaces of low degrees}. Amer. J. Math. 119 (1997), no. 5, 1139-1172.

 \bibitem[Voi96]{voi96} Claire Voisin: {\it On a conjecture of 
 Clemens on rational curves on hypersurfaces}. J. Differential Geom.
 44 (1996), no. 1, 200-213.
 
 \bibitem[Xu94]{xu94}Geng Xu: {\it Subvarieties of general hypersurfaces in projective space}. J. Differential Geom. 39 (1994), no. 1, 139-172.

 \bibitem[Zai87]{zai87} Mikhail Zaidenberg: {\it The complement to a general hypersurface of degree $2n$ in 
 $\bC\bP^n$ is not hyperbolic}. (Russian) Sibirsk. Mat. Zh. 28 (1987), no. 3, 91-100, 222. 

 \bibitem[Zai89]{zai89} Mikhail Zaidenberg: {\it Stability of hyperbolic embeddedness and the construction of examples.}
 (Russian) Mat. Sb. (N.S.) 135(177) (1988), no. 3, 361--372, 415; translation in Math. USSR-Sb. 63 (1989), no. 2, 351-361.

 \bibitem[Zai92]{zai92} Mikhail Zaidenberg: {\it Hyperbolicity in 
 projective spaces}.
 International Symposium ``Holomorphic Mappings, Diophantine Geometry
 and Related Topics'' (Kyoto, 1992). Surikaisekikenkyusho Kokyruoku  No. 819  (1993), 136-156. 

 \bibitem[Zai09]{zai09} Mikhail Zaidenberg: {\it Hyperbolicity of generic deformations}. (Russian) Funktsional. Anal. i Prilozhen. 43 (2009), no. 2, 39--46; translation in Funct. Anal. Appl. 43 (2009), no. 2, 113-118.
\end{thebibliography}
\end{document}